\documentclass[12pt]{amsart}
\usepackage{amsmath,amsthm,amsfonts,amssymb,mathrsfs}
\date{\today}

\usepackage{color}
\input xymatrix
\xyoption{all}

\usepackage{hyperref}

\newtheorem{theorem}{Theorem}
\newtheorem*{claim}{Claim}

\newtheorem{corollary}[theorem]{Corollary}
\newtheorem{lemma}[theorem]{Lemma}
\theoremstyle{definition}
\newtheorem{remark}[theorem]{Remark}

\newcommand{\N}{\mathbb N}

\def\N{\mathbb N}

\begin{document}

\title[On locally compact topological graph inverse semigroups]{On locally compact topological graph inverse semigroups}

\author[S.~Bardyla]{Serhii~Bardyla}
\address{Faculty of Mathematics, National University of Lviv,
Universytetska 1, Lviv, 79000, Ukraine}
\email{sbardyla@yahoo.com}



\begin{abstract}
In this paper we characterise graph inverse semigroups which admit only discrete locally compact semigroup topology. This characterization provides a complete answer on the question of Z.~Mesyan, J.~D.~Mitchell, M.~Morayne and Y.~H.~P\'{e}resse posed in~\cite{Mesyan-Mitchell-Morayne-Peresse-2013}.
\end{abstract}
\maketitle

In this paper all topological spaces are assumed to be Hausdorff. We shall follow the terminology of~\cite{Clifford-Preston-1961-1967,
Engelking-1989, Lawson-1998, Ruppert-1984}. By $\mathbb{N}$ we denote a set of all positive integers.
A semigroup $S$ is called an \emph{inverse semigroup} if for each element $a\in S$ there exists a unique
element $a^{-1}\in S$ such that
\begin{equation*}
    aa^{-1}a=a \qquad \mbox{and} \qquad a^{-1}aa^{-1}=a^{-1}.
\end{equation*}
A map which associates to any element of an inverse semigroup its
inverse is called the \emph{inversion}.

A directed graph $E=(E^{0},E^{1},r,s)$ consists of sets $E^{0},E^{1}$ of {\em vertices} and {\em edges}, respectively, together with functions $s,r:E^{1}\rightarrow E^{0}$ which are called {\em source} and {\em range}, respectively. In this paper we shall refer to directed graph as simply "graph".  A path $x=e_{1}\ldots e_{n}$ in graph $E$ is a finite sequence of edges $e_{1},\ldots,e_{n}$ such that $r(e_{i})=s(e_{i+1})$ for each positive integer $i<n$. We extend functions $s$ and $r$ on the set $\operatorname{Path}(E)$ of all pathes in graph $E$ by the following way: for each $x=e_{1}\ldots e_{n}\in \operatorname{Path}(E)$ put $s(x)=s(e_{1})$ and $r(x)=r(e_{n})$. By $|x|$ we denote the length of a path $x$. Observe that each vertex is a path which has zero length. Edge $e$ is called a {\em loop} if $s(e)=r(e)$. A path $x$ is called a {\em cycle} if $s(x)=r(x)$. Graph $E$ is called {\em finite} if the sets $E^0$ and $E^1$ are finite and {\em infinite} in the other case.

A topological (inverse) semigroup is a Hausdorff topological space together with a continuous semigroup operation (and an~inversion, respectively).  If $S$ is a semigroup (an inverse semigroup) and $\tau$ is a topology on $S$ such that $(S,\tau)$ is a topological (inverse) semigroup, then we
shall call $\tau$ a (\emph{inverse}) \emph{semigroup} \emph{topology} on $S$.

A bicyclic monoid ${\mathscr{C}}(p,q)$ is the semigroup with the identity $1$ generated by two elements $p$ and $q$ subject to the condition $pq=1$.
The bicyclic semigroup admits only the discrete semigroup topology~\cite{Eberhart-Selden-1969}. In \cite{Bertman-West-1976} this result was extended over the case of semitopological semigroups. The closure of a bicyclic semigroup in a locally compact topological inverse semigroup was described  in~\cite{Eberhart-Selden-1969}. A locally compact semitopological bicyclic monoid with adjoined zero is either compact or discrete space~\cite{Gutik-2015}. The problem of an embedding of the bicyclic monoid into compact-like topological semigroups discussed in \cite{Anderson-Hunter-Koch-1965, BanDimGut-2009, BanDimGut-2010, GutRep-2007,  Hildebrant-Koch-1988}.

One of the generalizations of the bicyclic semigroup is a $\lambda$-polycyclic monoid.
For a non-zero cardinal $\lambda$, $\lambda$-polycyclic monoid $\mathcal{P}_\lambda$ is the semigroup with identity and zero given by the presentation:
\begin{equation*}
    \mathcal{P}_\lambda=\left\langle \left\{p_i\right\}_{i\in\lambda}, \left\{p_i^{-1}\right\}_{i\in\lambda}\mid  p_i^{-1}p_i=1, p_j^{-1}p_i=0 \hbox{~for~} i\neq j\right\rangle.
\end{equation*}

Polycyclic monoid $\mathcal{P}_{k}$ over a finite cardinal $k$ was introduced in \cite{Nivat-Perrot-1970}. Observe that the bicyclic semigroup with adjoined zero is isomorphic to the polycyclic monoid $\mathcal{P}_{1}$. Algebraic properties of a semigroup $\mathcal{P}_{k}$ were investigated in \cite{Lawson-2009} and \cite{Meakin-1993}. Algebraic and topological properties of the $\lambda$-polycyclic monoid were investigated in~\cite{BardGut-2016(1)} and~\cite{BardGut-2016(2)}. In particular, it was proved that for every non-zero cardinal $\lambda$ the only locally compact semigroup topology on the $\lambda$-polycyclic monoid is a discrete topology. In~\cite{Bardyla-2016(1)} it was showed that a locally compact semitopological $\lambda$-polycyclic monoid is either compact or discrete space.

For a given directed graph $E=(E^{0},E^{1},r,s)$ a graph inverse semigroup $G(E)$ over a graph $E$ is a semigroup with zero generated by the sets $E^{0}$, $E^{1}$ together with a set $E^{-1}=\{e^{-1}:e\in E^{1}\}$ satisfying the following relations for all $a,b\in E^{0}$ and $e,f\in E^{1}$:
 \begin{itemize}
 \item [(i)]  $a\cdot b=a$ if $a=b$ and $a\cdot b=0$ if $a\neq b$;
 \item [(ii)] $s(e)\cdot e=e\cdot r(e)=e;$
 \item [(iii)] $e^{-1}\cdot s(e)=r(e)\cdot e^{-1}=e^{-1};$
 \item [(iv)] $e^{-1}\cdot f=r(e)$ if $e=f$ and $e^{-1}\cdot f=0$ if $e\neq f$.
\end{itemize}

Graph inverse semigroups are a generalization of the polycyclic monoids. In particular, for every non-zero cardinal $\lambda$, $\lambda$-polycyclic monoid is isomorphic to the graph inverse semigroup over the graph $E$ which consists of one vertex and $\lambda$ distinct loops.

By \cite[Lemma~2.6]{Jones-2011} each non-zero element of a graph inverse semigroup $G(E)$ is of the form  $uv^{-1}$ where $u,v\in \operatorname{Path}(E)\hbox{ and } r(u)=r(v)$. A semigroup operation in $G(E)$ is defined by the following way:
\begin{equation*}
\begin{split}
  &  u_1v_1^{-1}\cdot u_2v_2^{-1}=
    \left\{
      \begin{array}{ccl}
        u_1wv_2^{-1}, & \hbox{if~~} u_2=v_1w & \hbox{for some~} w\in \operatorname{Path}(E);\\
        u_1(v_2w)^{-1},   & \hbox{if~~} v_1=u_2w & \hbox{for some~} w\in \operatorname{Path}(E);\\
        0,              & \hbox{otherwise},
      \end{array}
    \right.
    \end{split}
\end{equation*}
 and $uv^{-1}\cdot 0=0\cdot uv^{-1}=0\cdot 0=0.$

Simply verifications show that $G(E)$ is an inverse semigroup, moreover, $(uv^{-1})^{-1}=vu^{-1}$.

Graph inverse semigroups play an important role in the study of rings and $C^{*}$-algebras (see \cite{Abrams-2005,Ara-2007,Cuntz-1980,Kumjian-1998,Paterson-1999}).
Algebraic properties of graph inverse semigroups were studied in \cite{Amal-2016, Bardyla-2017, Jones-2011, Jones-Lawson-2014, Lawson-2009, Mesyan-2016}. In \cite{Mesyan-Mitchell-Morayne-Peresse-2013} Z. Mesyan, J. D. Mitchell, M. Morayne and Y. H. P\'{e}resse investigated topological properties of graph inverse semigroups. In particular, one of the main results from~\cite{Mesyan-Mitchell-Morayne-Peresse-2013} is the following:
 \begin{theorem}[{\cite[Theorem 10]{Mesyan-Mitchell-Morayne-Peresse-2013}}]\label{th}
 If $E$ is a finite graph, then the only locally compact Hausdorff semigroup topology on $G(E)$ is the discrete topology.
 \end{theorem}
Also in~\cite{Mesyan-Mitchell-Morayne-Peresse-2013} the authors asked the following natural question:
{\it Can the Theorem~\ref{th} be generalized to all graphs?}

In this paper we give a complete answer on the above question. In particular, we describe all graph inverse semigroups which allow only discrete locally compact semigroup topology and propose a method of constructing a non-discrete locally compact metrizable inverse semigroup topology on graph inverse semigroups.

\section{Main result}
We shall say that a graph inverse semigroup $G(E)$ satisfies the condition $(\star)$ if for each countable subset $A=\{x_n\}_{n\in\mathbb{N}}\subset \operatorname{Path}(E)$ there exists an infinite subset $B=\{x_{n_k}\}_{k\in\mathbb{N}}\subset A$ and an element $\mu\in G(E)$ such that $\mu\cdot x_{n_{k}}\in \operatorname{Path}(E)$ and $|\mu\cdot x_{n_{k}}|>|x_{n_{k}}|$, for each $k\in\mathbb{N}$. Each graph inverse semigroup $G(E)$ over a finite graph $E$ satisfies the condition $(\star)$ (see~\cite[Lemma~9]{Mesyan-Mitchell-Morayne-Peresse-2013}).

\begin{remark}\label{rem}
If a graph inverse semigroup $G(E)$ satisfies the condition $(\star)$ then the set $E^0$ of all vertices of graph $E$ is finite. Indeed, if the set $E^0$ is infinite then for an arbitrary countable infinite subset $A\subseteq E^0\subseteq \operatorname{Path}(E)$ and for each element $\mu\in G(E)$ there exists a vertex $f\in A\subseteq E^0$ such that $\mu\cdot f=0\notin \operatorname{Path}(E)$ which implies that semigroup $G(E)$ does not satisfy the condition $(\star)$.
However, there exists a semigroup $G(E)$ over an infinite graph $E$ which satisfies the condition $(\star)$. In particular, each $\lambda$-polycyclic monoid satisfies the condition $(\star)$. Moreover, if for each vertex $e$ of a graph $E$ there exists a cycle $u\in \operatorname{Path}(E)$ such that $s(u)=r(u)=e$ and graph $E$ contains a finite amount of vertices then graph inverse semigroup $G(E)$ satisfies the condition $(\star)$.
\end{remark}

\begin{lemma}\label{triv}
Let $X$ be a first countable Hausdorff topological space which has only one non-isolated point $y$. Then $X$ a is metrizable space.
\end{lemma}

\begin{proof}
Fix any countable open neighborhood base $\mathcal{B}_y=\{U_n\}_{n\in\N}$ of the point $y$. For each $x\in X\setminus\{y\}$ put $V_x=\{x\}$. Simply verifications show that family $\mathcal{B}=\{V_x\}_{x\in X\setminus \{y\}}\cup \mathcal{B}_y$ is a $\sigma$-locally finite base of the topological space $X$. Since a family $\mathcal{B}$ consists of closed and open subsets of $X$ the space $X$ is regular. Hence by Nagata-Smirnov Metrization Theorem (see~\cite[Theorem 4.4.7]{Engelking-1989}) the space $X$ is metrizable.
\end{proof}

\begin{theorem}\label{main}
Discrete topology is the only locally compact semigroup topology on a graph inverse semigroup $G(E)$ if and only if $G(E)$ satisfies the condition $(\star)$.
\end{theorem}
\begin{proof}
$(\Leftarrow)$ Assuming the contrary, let $G(E)$ be a locally compact non-discrete topological graph inverse semigroup which satisfies the condition $(\star)$. Observe that by~\cite[Theorem 3]{Mesyan-Mitchell-Morayne-Peresse-2013} each non-zero element of $G(E)$ is an isolated point. Hence there exists an open neighborhood base of the point $0$ which consists of open compact subsets of $G(E)$. Since $G(E)\setminus\{0\}$ is a discrete subspace of $G(E)$ the local compactness of $G(E)$ implies that for each two open compact neighborhoods $U\subset V$ of the point $0$ the set $V\setminus U$ is finite. Hence element $0$ is a limit point of each infinite sequence $(x_n)_{n\in\N}$ which is contained in an arbitrary compact open neighborhood $U$ of $0$.

Fix an arbitrary compact open neighborhood $U$ of $0$. Then one of the following three conditions holds:
\begin{itemize}
\item[$(1)$] For each sequence $(x_n)_{n\in\N}\subset \operatorname{Path}(E)$, $\lim_{n\in\N} x_n\neq 0$ and $\lim_{n\in\N} x_n^{-1}\neq 0$.
\item[$(2)$] There exists a sequence $(x_n)_{n\in\N}\subset \operatorname{Path}(E)$ such that $\lim_{n\in\N} x_n=0$.
\item[$(3)$] There exists a sequence $(x_n)_{n\in\N}\subset \operatorname{Path}(E)$ such that $\lim_{n\in\N} x_n^{-1}=0$.
\end{itemize}
Suppose that condition $(1)$ holds. For each $y\in \operatorname{Path}(E)$ denote
$$X_{y}=\{x\in \operatorname{Path}(E):xy^{-1}\in U\setminus\{0\}\}.$$
We claim that for each element $y\in \operatorname{Path}(E)$ the set $X_{y}$ is finite. Indeed, suppose that there exists an element $y\in \operatorname{Path}(E)$ and an infinite subset $\{x_{n}\}_{n\in \mathbb{N}}\subset \operatorname{Path}(E)$ such that $x_ny^{-1}\in U\setminus\{0\}$ for each $n\in \mathbb{N}$. Observe that for any elements $a,b\in G(E)\setminus\{0\}$ the sets $$\{x\in G(E):a\cdot x=b\}\qquad \hbox{and}\qquad \{x\in G(E):x\cdot a=b\}$$ are finite (see~\cite[Lemma 1]{Mesyan-Mitchell-Morayne-Peresse-2013}). Hence the set $\{x_ny^{-1}:n\in\mathbb{N}\}$ is an infinite subset of $U$ which implies that $\lim_{n\in\mathbb{N}}x_ny^{-1}=0$. Since $0\cdot y=0$ the continuity of the semigroup operation in $G(E)$ implies that
$$0=\lim_{n\in\mathbb{N}}x_ny^{-1}=\lim_{n\in\mathbb{N}}x_ny^{-1}\cdot y=\lim_{n\in\mathbb{N}}x_n(y^{-1}y)=\lim_{n\in\mathbb{N}}x_n$$
which contradicts the assumption. Analogously it can be proved that for each $x\in \operatorname{Path}(E)$ the set $Y_{x}=\{y\in \operatorname{Path}(E):xy^{-1}\in U\setminus\{0\}\}$ is finite.

Since the set $U$ is infinite we see that the set $C=\{y\in \operatorname{Path}(E):X_{y}\neq\emptyset\}$ is infinite as well. Fix an arbitrary countable subset $B=\{y_n\}_{n\in\mathbb{N}}\subset C$. For each element $y_n\in B$ fix an element $x_{y_n}\in X_{y_{n}}$ such that $|x_{y_n}|\geq|x|$ for an arbitrary element $x\in X_{y_{n}}$. Since for each element $x\in \operatorname{Path}(E)$ the set $Y_x$ is finite we obtain that the set $\{x_{y_n}\}_{n\in\N}$ is infinite and $\lim_{n\in \mathbb{N}}x_{y_n}y_n^{-1}=0$. Since a semigroup $G(E)$ satisfies the condition $(\star)$ we can find an infinite subsequence $\{x_{y_{n_k}}\}_{k\in \mathbb{N}}\subset \{x_{y_n}\}_{n\in \mathbb{N}}$ and an element $\mu\in G(E)$ such that $\mu\cdot x_{y_{n_k}}\in \operatorname{Path}(E)$ and $|\mu\cdot x_{y_{n_k}}|>|x_{y_{n_k}}|$.
Since $\mu\cdot 0=0$ the continuity of the semigroup operation in $G(E)$ implies that $\lim_{k\in\mathbb{N}}\mu\cdot x_{y_{n_k}}y_{n_k}^{-1}=0$. Observe that $\mu\cdot x_{y_{n_k}}y_{n_k}^{-1}\neq 0$, because $\mu\cdot x_{y_{n_k}}\in\operatorname{Path}(E)$ and $r(\mu\cdot x_{y_{n_k}})=r(x_{y_{n_k}})=r(y_{n_k})$. Hence $\mu\cdot x_{y_{n_k}}\in X_{y_{n_k}}$ and $|\mu\cdot x_{y_{n_k}}|>|x_{y_{n_k}}|$ for some sufficiently large $k$ which contradicts the choice of $x_{y_{n_k}}$. Hence condition (1) does not hold.

Suppose that condition $(2)$ holds. Put $A=\operatorname{Path}(E)\cap U$.
Since there exists a sequence of elements from $\operatorname{Path}(E)$ which converges to $0$ the set $A$ is infinite.
For each $x\in \operatorname{Path}(E)$ put $$Y_x^{A}=\{y\in A:xy^{-1}\in U\setminus\{0\}\}=Y_x\cap A.$$ Clearly that for each $x\in \operatorname{Path}(E)$ the set $Y_x^{A}$ is finite. Put $D=\{x\in \operatorname{Path}(E):Y_x^{A}\neq\emptyset\}$.
Then one of the following two cases holds:
\begin{itemize}
\item[$(2.1)$] The set $D$ is finite.
\item[$(2.2)$] The set $D$ is infinite.
\end{itemize}
Consider case $(2.1)$. Finiteness of the set $D$ implies that the set $T=\{xy_x^{-1}:x\in D \hbox{ and } y_x\in Y_x^A\}$ is finite as well.
Then $V=U\setminus T$ is an open neighborhood of $0$ which satisfies the following condition: for each non-zero element $xy^{-1}\in V$, $y\notin A$. Since the set $E^0$ of all vertices of graph $E$ is finite (see~Remark~\ref{rem}) and the set $A$ is infinite there exists a vertex $e$ and an infinite subset $\widehat{A}\subset A$ such that for each element $y\in \widehat{A}$, $r(y)=e$. Fix an arbitrary element $y\in \widehat{A}$. Since $0\cdot y^{-1}=0$ the continuity of the semigroup operation in $G(E)$ implies that there exists an open compact neighborhood $W$ of $0$ such that $W\cdot y^{-1}\subset V$. Since $\widehat{A}$ is an infinite subset of $U$ we obtain that the set $W\cap \widehat{A}$ is infinite as well. Fix an arbitrary element $z\in W\cap \widehat{A}$. Then $z\cdot y^{-1}\in W\cdot y^{-1}\subset V\setminus\{0\}$ which contradicts the choice of the set $V$. Hence case~$(2.1)$ does not hold.

Consider case $(2.2)$.
Fix an arbitrary infinite countable subset $H=\{x_n\}_{n\in\mathbb{N}}$ of $D$. For each element $x_n\in H$ by $y_{x_n}$ we denote an element of $Y_{x_n}^{A}$ such that $|y_{x_n}|\geq|y|$ for an arbitrary element $y\in Y_{x_n}^{A}$. The set $B=\{y_{x_n}:n\in\N\}$ is infinite, because for each element $y\in \operatorname{Path}(E)$ the set $X_{y}$ is finite. Observe that $\lim_{n\in\mathbb{N}}x_{n}y_{x_n}^{-1}=0$.
Since semigroup $G(E)$ satisfies the condition $(\star)$ there exists an element $\mu\in G(E)$ and an infinite subset $C=\{y_{x_{n_k}}:k\in\N\}\subset B$ such that $\mu\cdot y_{x_{n_k}}\in \operatorname{Path}(E)$ and $|\mu\cdot y_{x_{n_k}}|>|y_{x_{n_k}}|$, for each $k\in \mathbb{N}$.
We claim that the set $(\mu\cdot C)\cap U$ is infinite. Indeed, since $\mu\cdot 0=0$ the continuity of the semigroup operation in $G(E)$ implies that there exists an open compact neighborhood $W$ of $0$ such that $\mu\cdot W\subset U$. Since for a fixed non-zero elements $a,b$ of $G(E)$ the set $\{x\in G(E): a\cdot x=b\}$ is finite and the set $W\cap C$ is infinite we obtain that the set $(\mu\cdot C)\cap U$ is infinite.

Since $0\cdot \mu^{-1}=0$ the continuity of the semigroup operation in $G(E)$ implies that
$$\lim_{k\in\mathbb{N}} x_{n_k}y_{x_{n_k}}^{-1}\cdot \mu^{-1}=\lim_{k\in\mathbb{N}} x_{n_k}(\mu\cdot y_{x_{n_k}})^{-1}=0.$$
Observe that for each $k\in\N$, $x_{n_k}(\mu\cdot y_{x_{n_k}})^{-1}\neq 0$, because $\mu\cdot y_{x_{n_k}}\in \operatorname{Path}(E)$ and $r(\mu\cdot y_{x_{n_k}})=r(y_{x_{n_k}})=r(x_{n_k})$. Hence there exists a positive integer $k$ such that $\mu\cdot y_{x_{n_k}}\in Y_{x_{n_k}}^{A}$ and $|\mu\cdot y_{x_{n_k}}|>|y_{x_{n_k}}|$ which contradicts the choice of $y_{x_{n_k}}$. The obtained contradiction implies that case~$(2.2)$ is not possible. Hence condition~$(2)$ does not hold.


Similar arguments imply that condition $(3)$ does not hold as well. The obtained contradiction implies that the only locally compact semigroup topology on a graph inverse semigroup $G(E)$ which satisfies the condition $(\star)$ is the discrete topology.

$(\Rightarrow)$
Suppose that a semigroup $G(E)$ does not satisfy the condition $(\star)$.
Then there are two cases to consider.
\begin{itemize}
\item[(1)] Graph $E$ contains infinitely many vertices.
\item[(2)] Graph $E$ contains finitely many vertices.
\end{itemize}
Consider case (1). Fix an arbitrary countable infinite subset $A=\{e_n: n\in\N\}\subset E^0$. Define a topology $\tau$ on the semigroup $G(E)$ by the following way:
each non-zero element of $G(E)$ is an isolated point and a family $\mathcal{B}=\{U_n: n\in\N\}$ forms an open neighborhood base of the point $0$, where $U_n=\{e_k\in A: k>n\}\cup\{0\}$. By Lemma~\ref{triv}, topological space $(G(E),\tau)$ is metrizable. Since for each $n\in\N$ the set $U_n$ is compact we obtain that $(G(E),\tau)$ is a locally compact space. Observe that $U_n^{-1}=U_n$ and hence inversion is continuous in $(G(E),\tau)$.
Let $U_n$ be an arbitrary open neighborhood of $0$ and $ab^{-1}\in G(E)\setminus\{0\}$. Then the continuity of the semigroup operation in $(G(E),\tau)$ can be derived from the following equalities:
\begin{itemize}
\item $U_n\cdot U_n=U_n$;
\item $ab^{-1}\cdot (U_n\setminus\{s(b)\})=0\subset U_n$;
\item $(U_n\setminus\{s(a)\})\cdot ab^{-1}=0\subset U_n$.
\end{itemize}
Hence $(G(E),\tau)$ is a locally compact metrizable topological inverse semigroup.

Consider case (2). Since semigroup $G(E)$ does not satisfy the condition $(\star)$ there exists an infinite subset $A=\{x_n\}_{n\in\mathbb{N}}\subset \operatorname{Path}(E)$ such that for each infinite subset $B=\{x_{n_k}\}_{k\in\mathbb{N}}\subset A$ and for each element $\mu\in G(E)$ there exists an element $x_{n_k}\in B$ such that either $\mu\cdot x_{n_k}\notin \operatorname{Path}(E)$ or $|\mu\cdot x_{n_k}|\leq|x_n|$. Observe that with no loss of generality we can assume that for each element $x_n\in A$, $s(x_n)=e_0$ for some fixed vertex $e_0\in E^0$. We claim the following:
\begin{claim}
There exists a pair of vertices $e_n$ and $e_{n+1}$ such that the set $\{u\in \operatorname{Path}(E): r(u)=e_n\}$ is finite and the set $\{a\in E^1:s(a)=e_n \hbox{ and } r(a)=e_{n+1}\}$ is infinite.
\end{claim}
\begin{proof}

{\bf Step 1}.
Observe that the choice of the set $A$ implies that there exists no edge which range is a vertex $e_0$.
Since the set of all vertices of the graph $E$ is finite and $A$ is an infinite set there exists a vertex $e_1$ and an infinite subset $B_1\subset A$ such that each element $x_n\in B_1$ is of the form $x_n=a_{(1,n)}v$ where $a_{(1,n)}$ is an edge such that $s(a_{(1,n)})=e_0$, $r(a_{(1,n)})=e_1$ for each $n\in\N$ and $s(v)=e_1$.
If the set $T_1=\{a_{(1,n)}:n\in\N\}$ is infinite then put $e_n=e_0$, $e_{n+1}=e_1$ and proof of the claim ends.

{\bf Step 2}. If the set $T_1=\{a_{(1,n)}:n\in\N\}$ is finite then similar arguments imply that there exists a vertex $e_2$ and an infinite subset $B_2\subset B_1$ such that each element $x_n\in B_2$ is of the form $x_n=a_1a_{(2,n)}v$, where $a_1$ is a fixed element from $T_1$, $a_{(2,n)}$ is an edge such that $s(a_{(2,n)})=e_1$, $r(a_{(2,n)})=e_2$ for each positive integer $n$ and $s(v)=e_2$.
We claim that $|y|\leq 1$ for each path $y\in \operatorname{Path}(E)$ such that $r(y)=e_1$. Indeed, if there exists an element $y\in \operatorname{Path}(E)$ such that $r(y)=e_1$ and $|y|>1$ then for an infinite subset $B_2\subset A$ there exists an element $\mu=ya_1^{-1}\in G(E)$ such that
$\mu\cdot B_2\subset \operatorname{Path}(E)
$
and for each element $x_k=a_1a_{(2,n)}v\in B_2$, $$|\mu\cdot a_1a_{(2,n)}v|=|ya_{(2,n)}v|>|a_1a_{(2,n)}v|$$ which contradicts the choice of the set $A$. If there exists a vertex $f$ such that the set $\{a\in E^1:s(a)=f\hbox{ and } r(a)=e_1\}$ is infinite then put $e_n=f$, $e_{n+1}=e_1$ and proof of the claim ends, because there are no edges which range is a vertex $f$ (in the other case there exists a path $y$ such that $r(y)=e_1$ and $|y|>1$ which contradicts the above arguments).
If for each vertex $f$ the set $\{a\in E^1:s(a)=f\hbox{ and } r(a)=e_1\}$ is finite then the set $\{y\in \operatorname{Path}(E):r(y)=e_1\}$ is finite as well, because graph $E$ contains a finite amount of vertices.
If the set $T_2=\{a_{(2,n)}:n\in\N\}$ is infinite then put $e_{n}=e_1$, $e_{n+1}=e_2$ and proof of the claim ends.

Suppose that we have done $k-1$ steps and do not find a desired pair of vertices.

{\bf Step k}. After the previous steps we obtain two finite sequence of sets $A\supset B_1\supset B_2\supset\ldots \supset B_{k-1}$;  $T_1,T_2,\ldots,T_{k-1}$ and a finite sequence of distinct vertices $e_1,\ldots e_{k-1}$ where for every $i<k$ each set $T_i$ is finite, each set $B_i$ is infinite and each element $a_{(i,n)}\in T_i$ is an edge such that $s(a_{(i,n)})=e_{i-1}$ and $r(a_{(i,n)})=e_{i}$. Each element $x_k\in B_{k-1}$ is of the form $x_k=a_1a_2\ldots a_{k-2}a_{(k-1,n)}v$ where $a_i$ is a fixed element from $T_i$ for each $i<k-1$, $a_{(k-1,n)}\in T_{k-1}$ for each $n\in\N$ and $s(v)=e_{k-1}$. The set $T_{k-1}=\{a_{(k-1,n)}:n\in\N\}$ is finite, because in the other case we would find a desired pair of vertices on the previous step. Hence there exists a vertex $e_k$ and an infinite subset $B_k\subset B_{k-1}$ such that each element $x_n\in B_k$ is of the form $x_n=a_1a_2\ldots a_{k-2}a_{k-1}a_{(k,n)}v$, where $a_{k-1}$ is a fixed element from $T_{k-1}$, $a_{(k,n)}$ is an edge such that $s(a_{(k,n)})=e_{k-1}$, $r(a_{(k,n)})=e_k$ for each positive integer $n$ and $s(v)=e_k$.
We claim that $|y|\leq k-1$ for each path $y\in \operatorname{Path}(E)$ such that $r(y)=e_{k-1}$. Indeed, if there exists an element $y\in \operatorname{Path}(E)$ such that $r(y)=e_{k-1}$ and $|y|>k-1$ then for an infinite subset $B_k\subset A$ there exists an element $\mu=y(a_1\ldots a_{k-1})^{-1}\in G(E)$ such that
$\mu \cdot B_k\subset \operatorname{Path}(E)
$
 and for each element $x_n=a_1\ldots a_{k-1}a_{(k,n)}v\in B_k$, $$|\mu\cdot a_1\ldots a_{k-1}a_{(k,n)}v|=|ya_{(k,n)}v|>|a_1\ldots a_{k-1}a_{(k,n)}v|$$ which contradicts the choice of the set $A$.

For each positive integer $i<k$ by $P_{i}$ we denote the set of all vertices $f$ such that there exists a path $u\in\operatorname{Path}(E)$ such that $s(u)=f$, $r(u)=e_{k-1}$ and $|u|=i$.

 If there exist a vertex $f\in P_{k-1}$ and a vertex $g$ such that the set
$$\{a\in E^1: s(a)=f \quad\hbox{and}\quad r(a)=g\}$$
is infinite then put $e_n=f$, $e_{n+1}=g$ and proof of the claim ends, because there are no pathes which range is vertex $f\in P_{k-1}$ (in the other case there exists a path $y$ such that $r(y)=e_{k-1}$ and $|y|>k-1$ which contradicts the above arguments).

If for each vertex $g$, for each vertex $f\in P_{k-1}$ the set $\{a\in E^1:s(a)=f \hbox{ and } r(a)=g\}$ is finite then the set $\{u\in \operatorname{Path}(E):r(u)\in P_{k-2}\}$ is finite as well.
If there exist a vertex $f\in P_{k-2}$ and a vertex $g$ such that the set
$$\{a\in E^1: s(a)=f \quad\hbox{and}\quad r(a)=g\}$$
is infinite then put $e_n=f$, $e_{n+1}=g$ and proof of the claim ends.

If for each vertex $g$, for each vertex $f\in P_{k-2}$ the set $\{a\in E^1:s(a)=f \hbox{ and } r(a)=g\}$ is finite then the set $\{u\in \operatorname{Path}(E):r(u)\in P_{k-3}\}$ is finite as well. We repeat our arguments and either find a desired pair of vertices or obtain that the set $\{u\in \operatorname{Path}(E):r(u)=e_{k-1}\}$ is finite.

If the set $\{u\in \operatorname{Path}(E):r(u)=e_{k-1}\}$ is finite and the set $T_k=\{a_{(k,n)}:n\in\N\}$ is infinite then put $e_n=e_{k-1}$, $e_{n+1}=e_k$ and proof of the claim ends.
If the set $T_k$ is finite we do the next step. Since the set $E^0$ of all vertices of graph $E$ is finite and the set $A$ is infinite we obtain that after a finite amount of steps we find a pair of vertices $e_n$ and $e_{n+1}$ such that the set $\{u\in \operatorname{Path}(E):r(u)=e_n\}$ is finite and the set $\{a\in E^1:s(a)=e_n \hbox{ and } r(a)=e_{n+1}\}$ is infinite.
\end{proof}
Fix those pair of vertices $e_n$ and $e_{n+1}$. Denote $L=\{a\in E^1:s(a)=e_n,r(a)=e_{n+1}\}$.
Let $M=\{a_n:n\in\N\}$ be an arbitrary countable infinite subset of $L$.
Put $$N=\{ua_na_n^{-1}v^{-1}:\hbox{ } u,v\in \operatorname{Path}(E) \hbox{ such that } r(u)=r(v)=e_n\hbox{ and }a_n\in M\}.$$

Define a topology $\tau$ on the semigroup $G(E)$ by the following way:
each non-zero element of $G(E)$ is an isolated point and a family $\mathcal{B}=\{U_n: n\in\N\}$ forms an open neighborhood base of the point $0$, where $U_n=\{ua_ka_k^{-1}v^{-1}\in N: k>n\}\cup\{0\}$.
Since for a fixed element $a_k\in M$ the set $\{ua_ka_k^{-1}v^{-1}:u,v\in \operatorname{Path}(E) \hbox{ such that } r(u)=r(v)=e_n\}$ is finite each basic open neighborhood of $0$ is compact. Hence $(G(E),\tau)$ is a locally compact topological space. By Lemma~\ref{triv}, $(G(E),\tau)$ is a metrizable topological space. The equality $U_n=U_n^{-1}$ provides the continuity of inversion in the semigroup $(G(E),\tau)$. To prove the continuity of the semigroup operation in $(G(E),\tau)$ we need to check it in the following three cases:
\begin{itemize}
\item[$(1)$] $bc^{-1}\cdot 0=0$;
\item[$(2)$] $0\cdot bc^{-1}=0$;
\item[$(3)$] $0\cdot 0=0$.
\end{itemize}
Consider case $(1)$.
We have the following three subcases:
\begin{itemize}
\item[$(1.1)$] $cd=u$ for some pathes $d,u\in\operatorname{Path}(E)$ such that $r(u)=e_n$;
\item[$(1.2)$] $c=ua_{n_0}d$ for some pathes $u,d\in\operatorname{Path}(E)$ such that $r(u)=e_n$ and $a_{n_0}\in M$;
\item[$(1.3)$] otherwise.
\end{itemize}

Consider subcase $(1.1)$. Fix an arbitrary open basic neighborhood $U_k$ of $0$. The continuity of the semigroup operation in $(G(E),\tau)$ can derived from the inclusion $bc^{-1}\cdot U_k\subset U_k$.

Consider subcase $(1.2)$. Fix an arbitrary open basic neighborhood $U_k$ of $0$. The continuity of the semigroup operation in $(G(E),\tau)$ can derived from the inclusion $bc^{-1}\cdot U_{\max\{n_0+1,k\}}=0\in U_k$.

If there exists no elements $d,u\in\operatorname{Path}(E)$ and $a_{n_0}\in M$ such that $r(u)=e_n$ and $cd=u$ or $c=ua_{n_0}d$ then
$bc^{-1}\cdot U_{1}=0\in U_k$.  Hence the semigroup operation in $(G(E),\tau)$ is continuous in case $(1)$.

Continuity of the semigroup operation in $(G(E),\tau)$ in case $(2)$ immediately follows from the continuity of the semigroup operation in $(G(E),\tau)$ in case $(1)$ and from the continuity of the inversion in $(G(E),\tau)$.

Consider case $(3)$. We claim that $U_k\cdot U_k\subseteq U_k$ for each $k\in\N$. Indeed, fix two arbitrary elements $u_1p_np_n^{-1}v_1^{-1}$ and $u_2p_mp_m^{-1}v_2^{-1}$ from $U_k$. Observe that $n>k$ and $m>k$. Then
$$u_1p_np_n^{-1}v_1^{-1}\cdot u_2p_mp_m^{-1}v_2^{-1}=0\in U_k \qquad \hbox{if either} \qquad v_1\neq u_2 \quad \hbox{or}\quad n\neq m$$
and
$$u_1p_np_n^{-1}v_1^{-1}\cdot u_2p_mp_m^{-1}v_2^{-1}=u_1p_np_n^{-1}v_2^{-1}\in U_k\qquad \hbox{if} \qquad v_1=u_2 \quad \hbox{and} \quad n=m.$$
Hence $(G(E),\tau)$ is a metrizable locally compact topological inverse semigroup.
 \end{proof}

\begin{corollary}
There exists a non-discrete locally compact semigroup topology on a graph inverse semigroup $G(E)$ if and only if there exists a non-discrete topology $\tau$ such that $(G(E),\tau)$ is a locally compact metrizable topological inverse semigroup.
\end{corollary}


Theorem~\ref{main} can be reformulated in the following way:
\begin{theorem}
Discrete topology is the only locally compact semigroup topology on the graph inverse semigroup $G(E)$ if and only if graph $E$ contains a finite amount of vertices and there does not exist a pair of vertices $e,f\in E^0$ such that the set $\{u\in \operatorname{Path}(E): r(u)=e\}$ is finite and the set $\{a\in E^1:s(a)=e \hbox{ and } r(a)=f\}$ is infinite.
\end{theorem}

\begin{proof}
If the set $E^0$ of all vertices of graph $E$ is infinite or there exists a pair a vertices $e,f\in E^0$ such that the set $\{u\in \operatorname{Path}(E): r(u)=e\}$ is finite and the set $\{a\in E^1:s(a)=e \hbox{ and } r(a)=f\}$ is infinite then we can introduce a non-discrete locally compact metrizable inverse semigroup topology $\tau$ on the semigroup $G(E)$ by the same way as in proof of the Theorem~\ref{main}.

If there exists a non-discrete locally compact semigroup topology on the graph inverse semigroup $G(E)$ then by Theorem~\ref{main} semigroup $G(E)$ does not satisfy the condition $(\star)$. Hence by the proof of the Theorem~\ref{main} either the set of all vertices of graph $E$ is infinite or graph $E$ contains a pair of vertices $e,f\in E^0$ such that the set $\{u\in \operatorname{Path}(E): r(u)=e\}$ is finite and the set $\{a\in E^1:s(a)=e \hbox{ and } r(a)=f\}$ is infinite.
\end{proof}



\end{document}